\documentclass[a4paper,leqno]{amsart}

\usepackage{amsmath}
\usepackage{amsfonts}
\usepackage{palatino}

\theoremstyle{plain}

\newtheorem{teo}{Theorem}[section]

\newtheorem{prop}[teo]{Proposition}

\newtheorem{ackn}{Acknowledgement\!}

\theoremstyle{definition}
\newtheorem{dfnz}[teo]{Definition}

\theoremstyle{remark}

\newtheorem{rem}[teo]{Remark}

\numberwithin{equation}{section}

\def\R{{{\mathbb R}}}

\def\RRR{{\mathrm R}}

\def\AAA{{\mathrm A}}
\def\HHH{{\mathrm H}}

\def\TTT{{\mathrm{T}}}
\def\comp{\circ}
\def\Ri{\R^{n+1}}

\begin{document}

\title[Remarks on Monotonicity Formula for MCF]
{Some Remarks on Huisken's Monotonicity Formula for Mean Curvature Flow}

\author[Annibale Magni]{Annibale Magni}
\address[Annibale Magni]{SISSA -- International School for Advanced
  Studies, Via Beirut 2--4, 
Trieste, Italy, 34014}
\email[A. Magni]{magni@sissa.it}

\author[Carlo Mantegazza]{Carlo Mantegazza}
\address[Carlo Mantegazza]{Scuola Normale Superiore di Pisa, P.za Cavalieri 7, Pisa, Italy, 56126}
\email[C. Mantegazza]{c.mantegazza@sns.it}

\date{10 January 2009}

\keywords{Mean Curvature Flow}

\begin{abstract}
We discuss a monotone quantity related to Huisken's
monotonicity formula and some technical consequences for mean curvature flow.
\end{abstract}

\maketitle

\tableofcontents

\section{Maximizing Huisken's Monotonicity Formula}

For an immersed hypersurface $M\subset\R^{n+1}$, we call $\AAA$ and $\HHH$
respectively its {\em second fundamental form} and its {\em mean
  curvature}.

Let $M_t=\varphi(M,t)$ be the {\em mean curvature flow} (MCF) of 
an $n$--dimensional compact hypersurface in
$\R^{n+1}$, defined by
the smooth family of immersions $\varphi:M\times[0,T)\to\R^{n+1}$
which satisfies $\partial_t\varphi=\HHH\nu$ 
where $\nu$ is the ``inner'' unit normal vector field to the hypersurface.

Huisken in~\cite{huisk3} found his fundamental monotonicity formula
\begin{equation}\label{monfor}
 \frac{d\,}{dt}\int_M\frac{e^{-\frac{\vert
  x-p\vert^2}{4(C-t)}}}{[4\pi(C-t)]^{n/2}}\,d\mu_t(x)=-\int_{M}\frac{e^{-\frac{\vert
      x-p\vert^2}{4(C-t)}}}{[4\pi(C-t)])^{n/2}}\left\vert\HHH+\frac{\langle x-p\,\vert\,\nu\rangle}{2(C-t)}
\right\vert^2\,d\mu_t(x)\leq 0\,,
\end{equation}
for every $p\in\Ri$, in the time interval $[0,\min\{C,T\})$. Here $d\mu_t$ is the
canonical measure on $M$ associated to the metric induced by the immersion at time $t$.\\
We call the quantity $\int_M\frac{e^{-\frac{\vert
  x-p\vert^2}{4(C-t)}}}{[4\pi(C-t)]^{n/2}}\,d\mu_t(x)$, the Huisken's functional.\\
Such formula was generalized by Hamilton in~\cite{hamilton7,hamilton8}
as follows, suppose that we have a positive smooth solution of $u_t=-\Delta u$ in
$\R^{n+1}\times [0,C)$ then, in the time interval $[0,\min\{C,T\})$, there holds
\begin{align}\label{huiskent}
  \frac{d\,}{dt}\,\Bigl[\,\sqrt{2(C-t)}\int_{M}u\,d\mu_t\,\Bigr]
  =&\,-\sqrt{2(C-t)}\int_{M}u\,\vert\HHH-\langle \nabla\log{u}\,\vert\,\nu\rangle\vert^2\,d\mu_t\\
  &\,-\sqrt{2(C-t)}\int_{M}\Bigl(\nabla^\perp\nabla^\perp u-\frac{\vert \nabla^\perp
    u\vert^2}{u}+\frac{u}{2(C-t)}\Bigr)\,d\mu_t\nonumber
\end{align}
where $\nabla^\perp$ denotes the covariant derivative along the normal
direction.

\begin{dfnz} Let $\varphi:M\to\R^{n+1}$ be a smooth, compact, immersed
  hypersurface.\\
Given $\tau>0$, we consider the family ${\mathcal
    F}_\tau$ of smooth positive functions $u:\R^{n+1}\to\R$ 
  such that $\int_{\R^{n+1}}u\,dx=1$ and 
  there exists a smooth positive solution of the problem
  $$
  \begin{cases}
    v_t=-\Delta v\, \text{ in $\R^{n+1}\times[0,\tau)$}\,,\\
    v(x,0)=u(x)\, \text{ for every $p\in\R^{n+1}$}\,.\\
  \end{cases}
  $$
  Then, we define the following quantity
  $$
  \sigma(\varphi,\tau)=\sup_{u\in{\mathcal{F}_\tau}} \sqrt{4\pi\tau}\int_{M}u\,d\mu\,.
  $$
\end{dfnz}

\begin{rem} 
The heat kernel $K_{\R^{n+1}}(x,p,\tau)=\frac{e^{-\frac{\vert
  x-p\vert^2}{4\tau}}}{(4\pi\tau)^{(n+1)/2}}$ of $\R^{n+1}$ at time $\tau>0$ and
  point $p\in\R^{n+1}$ clearly belongs to the family ${\mathcal
  F}_\tau$.
\end{rem}

It is immediate to see by this remark that 
the quantity $\sigma(\varphi,\tau)$ is positive and
  precisely, for every $p\in\R^{n+1}$ and $\tau>0$,
$$
\sigma(\varphi,\tau)\geq \sqrt{4\pi\tau}\int_{M}\frac{e^{-\frac{\vert
  x-p\vert^2}{4\tau}}}{(4\pi\tau)^{(n+1)/2}}\,d\mu(x)
=\int_{M}\frac{e^{-\frac{\vert
      x-p\vert^2}{4\tau}}}{(4\pi\tau)^{n/2}}\,d\mu(x)>0\,,
$$
which is the quantity of the ``classical'' Huisken's monotonicity
formula. Hence,
\begin{equation}
\sigma(\varphi,\tau)\geq \sup_{p\in\R^{n+1}}\int_{M}\frac{e^{-\frac{\vert
      x-p\vert^2}{4\tau}}}{(4\pi\tau)^{n/2}}\,d\mu(x)>0\,.
\end{equation}

We want to see that actually this inequality is an equality, that is,
we can take the $\sup$ only on heat kernels. Moreover, the $\sup$ is a
maximum.

We work out some properties of the functions
$u\in{\mathcal F}_\tau$.\\ 
We recall the integrated version of Li--Yau Harnack
inequality (see~\cite{liyau}).

\begin{prop}[Li--Yau integral Harnack inequality] 
Let $u:\R^{n+1}\times(0,T)\to\R$ be a smooth positive solution of heat
equation, then for every $0<t\leq s<T$ we have 
\begin{equation*}
u(x,t)\leq
u(y,s)\left(\frac{s}{t}\right)^{(n+1)/2}e^{\,\frac{\vert
    x-y\vert^2}{4(s-t)}}\,.
\end{equation*}
\end{prop}
Since the functions $v:\R^{n+1}\times[0,\tau)\to\R$ associated to any
$u\in{\mathcal F}_\tau$ are positive solutions of the backward heat equation, 
such inequality reads, for $0\leq s\leq t<\tau$,
\begin{equation*}
v(x,t)\leq
v(y,s)\left(\frac{\tau-s}{\tau-t}\right)^{(n+1)/2}e^{\,\frac{\vert
    x-y\vert^2}{4(t-s)}}\,.
\end{equation*}
This estimate, together with the uniqueness theorem for positive
solution of the heat equation (see again~\cite{liyau}), implies that
the function $u=v(\cdot,0)$ is obtained by convolution of the function
$v(\cdot,t)$ with the {\em forward} heat kernel at time $t>0$. This fact
implies that the condition $\int_{\R^{n+1}}v(x,t)\,dx=1$ holds for
every $t\in[0,\tau)$, and that every derivative of every function $v$
is bounded in the strip $[0,\tau-\varepsilon]$, for every
$\varepsilon>0$.

The functions $v(\cdot,t)$ weakly$^*$
  converge as probability measures, as $t\to\tau$, to some positive
  unit measure $\lambda$ on $\R^{n+1}$ such that
\begin{equation}\label{meas}
v(x,t)=\int_{\R^{n+1}}\frac{e^{-\frac{\vert
      x-y\vert^2}{4(\tau-t)}}}{[4\pi(\tau-t)]^{(n+1)/2}}\,d\lambda(y)\,.
\end{equation}
Conversely, every probability
  measure $\lambda$, by convolution with the heat kernel, gives rise to
  a function $v$ such that $v(\cdot,\tau)\in{\mathcal F}_\tau$, 
the most interesting case being $\lambda=\delta_p$ for
$p\in\R^{n+1}$.\\
Indeed, we know that for every $t\in[0,\tau)$ and $s\in(t,\tau)$,
$$
v(x,t)=\int_{\R^{n+1}}v(y,s)\frac{e^{\frac{\vert
      x-y\vert^2}{4(t-s)}}}{[4\pi(s-t)]^{(n+1)/2}}\,dx
$$
hence, choosing a sequence of times $s_i\nearrow\tau$ such that the
measures $v(\cdot,s_i){\mathcal L}^{n+1}$ weakly$^*$ converge to
some measure $\lambda$, we get equality~\eqref{meas}, 
since $\frac{e^{\frac{\vert x-y\vert^2}{4(t-s)}}}{[4\pi(s-t)]^{(n+1)/2}}$
converges uniformly to $\frac{e^{-\frac{\vert
      x-y\vert^2}{4(\tau-t)}}}{[4\pi(\tau-t)]^{(n+1)/2}}$ on
$\R^{n+1}$, as $s\to\tau$.\\
This representation formula also implies that the limit measure
$\lambda$ is unique and that actually
$\lim_{s\to\tau}v(\cdot,s){\mathcal L}^{n+1}=\lambda$ in the weak$^*$
convergence of measures on $\R^{n+1}$.\\
Finally, we show that $\vert\lambda\vert=1$. This follows by
Fubini--Tonelli's theorem for positive product measures, as
$\int_{\R^{n+1}}u(x)\,dx=1$,
\begin{align*}
1=&\,\int_{\R^{n+1}}u(x)\,dx=\int_{\R^{n+1}}\int_{\R^{n+1}}
\frac{e^{-\frac{\vert
    x-y\vert^2}{4\tau}}}{{[4\pi\tau]^{(n+1)/2}}}\,d\lambda(y)\,dx\\
=&\,\int_{\R^{n+1}}\int_{\R^{n+1}}\frac{e^{-\frac{\vert
    x-y\vert^2}{4\tau}}}{{[4\pi\tau]^{(n+1)/2}}}\,dx\,d\lambda(y)\\
=&\,\int_{\R^{n+1}}\,d\lambda(y)=\vert\lambda\vert\,.
\end{align*}
By this discussion it follows that the family ${\mathcal F}_\tau$
consists of
\begin{equation*}
u(x,t)=\int_{\R^{n+1}}\frac{e^{-\frac{\vert
      x-y\vert^2}{4\tau}}}{[4\pi\tau]^{(n+1)/2}}\,d\lambda(y)
\end{equation*}
where $\lambda$ varies among the convex set of Borel probability measures on
$\R^{n+1}$ (which is weak$^*$--compact).

A consequence of this fact is that since the integral
$\sqrt{4\pi\tau}\int_{M}u\,d\mu$ is a linear functional in the
function $u$, the $\sup$ in defining
$\sigma(\varphi,\tau)$ 
can be taken considering only the extremal points of the 
above convex, which are the {\em delta} measures in $\R^{n+1}$.
Consequently, the functions $u$ to be considered can be restricted to
be heat kernels at time $\tau>0$.\\

It is then easy to conclude that as the hypersurface $M$ is compact in
$\R^{n+1}$, the $\sup$ is actually a maximum.

\begin{prop}\label{propo} The quantity $\sigma(\varphi,\tau)$ is given by
$$
  \sigma(\varphi,\tau)=\max_{p\in\R^{n+1}} \int_{M}\frac{e^{-\frac{\vert
      x-p\vert^2}{4\tau}}}{(4\pi\tau)^{n/2}}\,d\mu(x)\,.
$$
\end{prop}

We have also easily that
\begin{equation*}
\sigma(\varphi,\tau)=\sup_{p\in\R^{n+1}}\int_{M}\frac{e^{-\frac{\vert
      x-p\vert^2}{4\tau}}}{(4\pi\tau)^{n/2}}\,d\mu(x)\leq
\int_{M}\frac{1}{(4\pi\tau)^{n/2}}\,d\mu(x)\leq
\frac{\mathrm{Area}(M)}{(4\pi\tau)^{n/2}}\,.
\end{equation*}

\begin{prop}[Rescaling Invariance] For every $\lambda>0$ we have
  $$
  \sigma(\lambda \varphi,{\lambda}^2\tau)=\sigma(\varphi,\tau)\,.
  $$
\label{rescprop}
\end{prop}
\begin{proof}
Let $u\in{\mathcal F}_\tau$ with associate solution of backward heat
equation $v:\R^{n+1}\times[0,\tau)\to\R$ and consider the rescaled function
$\widetilde{u}(y)=u(y/\lambda)\lambda^{-(n+1)}$.\\
It is easy to see that
$$
\int_{\R^{n+1}}\widetilde{u}(y)\,dy=\lambda^{-(n+1)}\int_{\R^{n+1}}u(y/\lambda)\,dy=
\int_{\R^{n+1}}u(x)\,dx=1
$$
with the change of variable $x=\lambda^{-(n+1)}y$, moreover the
function $\widetilde{v}(y,s)=v(y/\lambda,s/\lambda^2)\lambda^{-(n+1)}$ is a positive 
solution of the backward heat equation on the time interval
$\lambda^2\tau$, hence $\widetilde{u}\in{\mathcal
  F}_{\lambda^2\tau}$.\\
It is now a straightforward computation to see that
$$
\sqrt{4\pi\lambda^2\tau}\int_{M}\widetilde{u}\,d\mu_{\lambda\varphi}=
\sqrt{4\pi\tau}\int_{M}u\,d\mu_{\varphi}\,,
$$
for every smooth immersion of
a compact hypersurface $\varphi:M\to\R^{n+1}$. The statement clearly follows.
\end{proof}

By formula~\eqref{huiskent}, as the second term vanishes when $v$ is a
backward heat kernel, it follows that if $\varphi:M\times[0,T)\to\R^{n+1}$ is the MCF of a compact
hypersurface $M$, we have
\begin{align*}
  \frac{d\,}{dt}&\,\Bigl[\,\sqrt{2(C-t)}\int_{M} K_{\R^{n+1}}(x,p,C-t)\,d\mu_t(x)\,\Bigr]\\
  &\,=-\sqrt{2(C-t)}\int_{M}\,K_{\R^{n+1}}(x,p,C-t)\vert\HHH+\langle x-p\,\vert\,\nu\rangle/2(C-t)\vert^2\,d\mu_t(x)\nonumber
\end{align*}
which is clearly negative in the time interval $[0,\min\{C,T\})$.

\begin{prop}[Monotonicity and Differentiability]\label{monotsigma} 
  Along a MCF, $\varphi:M\times[0,T)\to\R^{n+1}$, if $\tau(t)=C-t$ for
  some constant $C>0$, the quantity 
  $\sigma(\varphi_t,\tau)$ is monotone nonincreasing in the time
  interval $[0,\min\{C,T\})$, hence it is differentiable almost
  everywhere.\\
  Moreover, letting $p_\tau$ a point in $\R^{n+1}$ such that
  $K_{\R^{n+1}}(x,p_\tau,\tau)$ is one of maximizer for
  $\sigma(\varphi_t,\tau(t))$ of Proposition~\ref{propo}, we have for almost every
  $t\in[0,\min\{C,T\})$,
  \begin{equation}
    \frac{d\,}{dt}\sigma(\varphi_t,\tau)\leq-\int_{M}\frac{e^{-\frac{\vert
      x-p_\tau\vert^2}{4\tau}}}{(4\pi\tau)^{n/2}}\left\vert\HHH+\frac{\langle x-p_\tau\,\vert\,\nu\rangle}{2\tau}
\right\vert^2\,d\mu_t
  \end{equation}
  or, since this inequality has to be intended in distributional
  sense, for every $0\leq r<t\leq\min\{C,T\}$,
  \begin{equation}
    \sigma(\varphi_r,\tau(r))-\sigma(\varphi_t,\tau(t))
    \geq \int_r^t\int_{M}\frac{e^{-\frac{\vert
      x-p_{\tau(s)}\vert^2}{4\tau(s)}}}{(4\pi\tau(s))^{n/2}}
\left\vert\HHH+\frac{\langle x-p_{\tau(s)}\,\vert\,\nu\rangle}{2\tau(s)}
\right\vert^2\,d\mu_s\,ds
\end{equation}
\end{prop}
\begin{proof} As the function $\sigma(\varphi_t,\tau)$ is the maximum
of monotone nonincreasing smooth functions, it also must be monotone
nonincreasing, hence, differentiable at almost every time
$t\in[0,\min\{C,T\})$.\\
The last assertion is standard, using {\em Hamilton's trick}
  (see~\cite{hamilton2}) to 
  exchange the $\sup$ and derivative operations.
\end{proof}

\begin{rem}
Notice that the quantity $\sigma$ can be defined also for any
$n$--dimensional countably rectifiable subset $S$ 
of $\R^{n+1}$, by substituting in the definition the
term $\int_M u\,d\mu$ with $\int_S u\,d{\mathcal H}^n$, where
${\mathcal H}^n$ is the $n$--dimensional Hausdorff measure (possibly {\em counting
  multiplicities}). If then $S$ is the support of a compact rectifiable {\em
  varifold}, with finite ${\mathrm {Area}}$, moving by mean curvature
according to Brakke's definition (see~\cite{brakke}),
Huisken's monotonicity formula~\eqref{huiskent} holds,
hence, also this proposition.
\end{rem}

\begin{dfnz}\label{sigdef}
We define, in the same hypothesis, for $\tau=C-t$ with $C\leq T$,
$$
\Sigma(C)=\lim_{t\to C^-}\sigma(\varphi_t,\tau)\,,
$$
and $\Sigma=\Sigma(T)$.
\end{dfnz}

By the previous discussion, 
$\Sigma\geq\sup_{p\in\R^{n+1}}{\Theta}(p)$, where this latter 
quantity is defined as 
\begin{equation}\label{theta}
\Theta(p)=\lim_{t\to T^-}\int_M\frac{e^{-\frac{\vert
  x-p\vert^2}{4(T-t)}}}{[4\pi(T-t)]^{n/2}}\,d\mu_t(x)\,,
\end{equation}
the existence of this limit for every $p\in\R^{n+1}$ is an obvious consequence
of Huisken's monotonicity formula.\\
Moreover, it is easy to see also the existence 
of $\max_{p\in\R^{n+1}}{\Theta}(p)$.

\begin{dfnz} Let $\varphi:M\to\R^{n+1}$ be a smooth, compact, immersed
  hypersurface. Then we define 
$$
\nu(\varphi)=\sup_{\tau>0}\sigma(\varphi,\tau)\,.
$$
\end{dfnz}

\begin{prop}\label{nudef}
The quantity $\nu(\varphi)$ is finite and actually reached by some
$\tau_\varphi$.
\end{prop}
\begin{proof}
Indeed, we have 
$$
\lim_{\tau\to0^+}\sigma(\varphi,\tau)=\Theta(\varphi)>0\,,
$$
where $\Theta(\varphi)$ is the maximum (which clearly exists as $M$ is
compact) of the $n$--dimensional density of $\varphi(M)$ in
$\R^{n+1}$. Then, if $\varphi$ is an embedding, $\Theta(\varphi)=1$,
otherwise it will be the highest multiplicity of the points of
$\varphi(M)$.\\
We show then that
$$
\lim_{\tau\to+\infty}\sigma(\varphi,\tau)=0\,.
$$
By the rescaling property of $\sigma$, we have 
$\sigma(\varphi,\tau)=\sigma(\varphi/\sqrt{4\pi\tau},1/4\pi)$, hence we need to show that
$$
\limsup_{\tau\to+\infty}\sup_{u\in{\mathcal{F}_1}}\int_{\frac{M}{\sqrt{4\pi\tau}}}u\,d\mu=0\,.
$$
But we already saw that any function $u\in{\mathcal{F}_1}$ satisfies
$0\leq u(x)\leq \frac{1}{(4\pi)^{(n+1)/2}}$ hence, 
\begin{equation*}
\limsup_{\tau\to+\infty}\sup_{u\in{\mathcal{F}_1}}\int_{\frac{M}{\sqrt{4\pi\tau}}}u\,d\mu
\leq \limsup_{\tau\to+\infty}\frac{{\mathrm{Vol}}(M/\sqrt{4\pi\tau})}
{(4\pi)^{(n+1)/2}}=\limsup_{\tau\to+\infty}\frac{{\mathrm{Vol}}(M)}
{(4\pi)^{(2n+1)/2}}\,\tau^{-n/2}=0\,.
\end{equation*}
\end{proof}

The following statement can be proved by the same argument of the proof of
Proposition~\ref{monotsigma}.

\begin{prop}[Monotonicity and Differentiability -- II]
  Along a MCF, $\varphi:M\times[0,T)\to\R^{n+1}$, 
  the quantity above $\nu(\varphi_t)$ is monotone non increasing in the time
  interval $[0,T)$, hence it is differentiable almost
  everywhere.\\
  Moreover, letting $p_\varphi\in\R^{n+1}$ and $\tau_\varphi$ to be 
  some of the maximizers whose
  existence is granted by Propositions~\ref{propo} and~\ref{nudef}, we have for almost every $t\in[0,T)$,
  \begin{equation}
    \frac{d\,}{dt}\nu(\varphi_t)\leq -\int_{M}\frac{e^{-\frac{\vert
      x-p_{\varphi_t}\vert^2}{4\tau_{\varphi_t}}}}{(4\pi\tau_{\varphi_t})^{n/2}}
\left\vert\HHH+\frac{\langle x-p_{\varphi_t}\,\vert\,\nu\rangle}{2\tau_{\varphi_t}}\right\vert^2\,d\mu_t(x)
  \end{equation}
  or, since this inequality has to be intended in distributional
  sense, for every $0\leq r<t<T$,
  \begin{equation}
      \nu(\varphi_r)-\nu(\varphi_t)
    \geq   \int_r^t\int_{M}\frac{e^{-\frac{\vert
      x-p_{\varphi_s}\vert^2}{4\tau_{\varphi_s}}}}{(4\pi\tau_{\varphi_s})^{n/2}}
\left\vert\HHH+\frac{\langle x-p_{\varphi_s}\,\vert\,\nu\rangle}{2\tau_{\varphi_s}}\right\vert^2\,d\mu_s(x)\,ds\,.
\end{equation}
\end{prop}

\begin{rem} One can repeat all this analysis for a compact, immersed
  hypersurface of a flat Riemannian manifold $\TTT$. Moreover, if the
  original hypersurface $\varphi:M\to\R^{n+1}$ is immersed in
  $\R^{n+1}$, we can choose a Riemannian covering map 
  $I:\R^{n+1}\to\TTT$ and consider the immersion
  $\widetilde{\varphi}=I\comp\varphi:M\to\TTT$. Then, we define as
  above, for every $\tau>0$, the family ${\mathcal
    F}_{\TTT,\tau}$ of smooth positive functions $u:\TTT\to\R$ 
  such that $\int_\TTT u\,dx=1$ and 
  there exists a smooth positive solution of the problem
  $$
  \begin{cases}
    v_t=-\Delta v\, \text{ in $\TTT\times[0,\tau)$}\\
    v(p,0)=u(x)\, \text{ for every $p\in\TTT$}\,.\\
  \end{cases}
  $$
  Then, we define the following quantity
  $$
  \sigma_\TTT(\varphi,\tau)=\sup_{u\in{\mathcal{F}_{\TTT,\tau}}}\sqrt{4\pi\tau}\int_{\widetilde{M}} 
  u\,d\widetilde{\mu}
  $$
where $\widetilde{M}$ refers to the fact that we are considering the
immersion $\widetilde{\varphi}:M\to\TTT$.

Notice that another possibility is simply to embed isometrically a convex set
$\Omega\subset\R^{n+1}$ containing $\varphi(M)$ in a flat Riemannian
manifold $\TTT$ (during the mean curvature flow a hypersurface
$\varphi$ initially contained in $\Omega$ stays ``inside'' for all the
evolution).

As before, these quantities are well defined, finite, positive and
  monotonically decreasing if $\varphi_t$ moves by mean curvature.
\end{rem} 

\section{Applications}

\subsection{A No--Breathers Result}

\begin{dfnz}
A {\em breather} (following Perelman~\cite{perel1}) 
for mean curvature flow in $\R^{n+1}$ is a smooth
 $n$--dimensional hypersurface evolving by mean
 curvature $\varphi:M\times[0,T)\to\R^{n+1}$, such that there exists a time
 $\overline{t}>0$, an isometry $L$ of $\R^{n+1}$ and a positive constant
 $\lambda\in\R$ with $\varphi(M, \overline{t})=\lambda L(\varphi(M,0))$.
\end{dfnz}
\begin{rem}
It is useless to consider nonshrinking (steady or expanding) 
compact breather of MCF, by the comparison with evolving spheres, 
they simply do not exist.\\
To authors' knowledge, it is unknown if there exist nonhomothetic,
noncompact ``steady'' or ``expanding'' breathers.
\end{rem}

\begin{teo}
Every compact breather is a homothetic solution to MCF.
\end{teo}
\begin{proof}
By the rescaling property of $\sigma$ in Proposition~\ref{rescprop}, fixing $C>0$ we
have
$$
\sigma(\varphi_0,C)\geq\sigma(\varphi_{\overline{t}},C-\overline{t})=\sigma(\lambda
\varphi_0,C-\overline{t})=\sigma(\varphi_0,(C-\overline{t})/\lambda^2)
$$
hence, if we choose $C=\frac{\overline{t}}{1-\lambda^2}$ we have
$C>\overline{t}$, as $\lambda<1$ and
$(C-\overline{t})/\lambda^2=C$. It follows that
$$
\sigma(\varphi_0,C)=\sigma(\varphi_{\overline{t}},C-\overline{t})
$$
and (for such special $C$), by Proposition~\ref{monotsigma}, if $\tau(t)=C-t$
$$
\int_0^{\overline{t}}\int_M
    \frac{e^{-\frac{\vert
      x-p_{\tau(t)}\vert^2}{4\tau(t)}}}{(4\pi\tau(t))^{n/2}}
\left\vert\HHH+\frac{\langle x-p_{\tau(t)}\,\vert\,\nu\rangle}{2\tau(t)}\right\vert^2\,d\mu_t\,dt=0\,.
$$
This implies that for every $t\in(0,\overline{t})$ we have 
$\HHH(x,t)=-\frac{\langle x-p_{\tau(t)}\,\vert\,\nu\rangle}{2(C-t)}$ for some
$p_{\tau(t)}\in\R^{n+1}$, which is the well known equation characterizing a
homothetically shrinking solution of MCF.
\end{proof}

\begin{rem} This is the same argument to show that compact
  shrinking breathers of Ricci flow are actually Ricci gradient solitons.\\
  Recalling the monotone nondecreasing quantity $\mu$ of Perelman
  in~\cite{perel1}, along a
  Ricci flow $g(t)$ of a compact, $n$--dimensional Riemannian manifold $M$, 
  $$
  \mu(g,\tau)=\inf_{\int_M u=1,\, u>0}\int_M \Bigl(\tau\Bigl[\RRR+\frac{\vert\nabla
    u\vert^2}{u}\Bigl]-u\log{u}-\frac{un}{2}\log{[4\pi\tau]}-un\Bigr)\,dV\,.
  $$
  By the rescaling property $\mu(\lambda g,\lambda\tau)=\mu(g,\tau)$, if
  we have that $g(\overline{t})=\lambda dL^*g(0)$ for some
  diffeomorphism $L:M\to M$ and $0<\lambda<1$, fixing $C>0$ we have
  $$
  \mu(g(0),C)\leq\mu(g(\overline{t}),C-\overline{t})=\mu(\lambda
  dL^*g(0),C-\overline{t})=\mu(\lambda g(0),C-\overline{t})
  =\mu(g(0),(C-\overline{t})/\lambda)
  $$
  hence, if we choose $C=\frac{\overline{t}}{1-\lambda}$ we have
  $C>\overline{t}$, as $\lambda<1$ and
  $(C-\overline{t})/\lambda=C$. It follows that
  $$
  \mu(g(0),C)=\mu(g({\overline{t}}),C-\overline{t})
  $$
  and by the results of Perelman, $g(t)$ is a shrinking soliton.
\end{rem}

\subsection{Singularities}

If $\varphi:M\times[0,T)\to\R^{n+1}$ is a MCF of a smooth, compact, {\em embedded} 
hypersurface, it is well known that during the flow it remains
embedded and there exists a finite maximal time $T>0$ of smooth existence when the
curvature $\AAA$ is unbounded as $t\nearrow T$.\\
Moreover for every $t\in[0,T)$
$$
\sup_{p\in M}\vert \AAA(p,t)\vert\geq\frac{1}{\sqrt{2(T-t)}}\,.
$$
If there exists a constant $C>0$ such that also
$$
\sup_{p\in M}\vert \AAA(p,t)\vert\leq\frac{C}{\sqrt{2(T-t)}}\,.
$$
we say that at $T$ we have a {\em type I} singularity, otherwise we say
the singularity is of {\em type II}.

We want to show that if at time $T$ we have a singularity, the associated quantity
$\Sigma=\lim_{t\to T^-}\sigma(\varphi_t,\tau)$ is larger than
one.\\
Indeed, for every $p\in\Ri$ such that there exists a sequence of
points $q_i\in M$ and times $t_i\nearrow T$ with
$p=\lim_{i\to\infty}\varphi(q_i,t_i)$, we consider the function
$\Theta(p)$ defined in equation~\eqref{theta}. 
By a simple semicontinuity argument, we can see that
$\Theta(p)\geq 1$ for every $p\in\Ri$ like above,
see~\cite[Corollary~4.20]{eck1}, hence, as 
$\Sigma\geq\sup_{p\in\R^{n+1}}{\Theta}(p)$ we get $\Sigma\geq 1$.\\
If then $\Sigma=1$, it forces $\Theta(p)=1$ for all such points $p$ which
implies, by the local regularity result of
White~\cite{white1}, that the flow cannot develop a singularity at
time $T$ (see also Ecker~\cite{eck1}).

Suppose now to have a type I singularity at time $T$.\\
By Proposition~\ref{monotsigma} we know that along this flow, for
$C=T$, hence, $\tau=T-t$,
\begin{equation*}
    \sigma(\varphi_r,T-r)-\sigma(\varphi_t,T-t)
    \geq \int_r^t\int_{M}\frac{e^{-\frac{\vert
      x-p_{T-s}\vert^2}{4(T-s)}}}{[4\pi(T-s)]^{n/2}}\left\vert\HHH+\frac{\langle x-p_{T-s}\,\vert\,\nu\rangle}{2(T-s)}
\right\vert^2\,d\mu_s(x)\,ds
\end{equation*}
for every $0\leq r\leq t\leq T$, hence,
\begin{equation}\label{f1}
    C(\varphi_0)\geq\sigma(\varphi_0,T)-\Sigma    \geq \int_0^T\int_{M}\frac{e^{-\frac{\vert
      x-p_{T-s}\vert^2}{4(T-s)}}}{[4\pi(T-s)]^{n/2}}\left\vert\HHH+\frac{\langle x-p_{T-s}\,\vert\,\nu\rangle}{2(T-s)}
\right\vert^2\,d\mu_s(x)\,ds
\end{equation}

Rescaling every hypersurface $\varphi_t$ as in~\cite{huisk3}, 
around the point $p_{T-t}$ as follows,
$$
\widetilde{\varphi}_s(q)=\frac{\varphi(q,t(s))-p_{T-t(s)}}{\sqrt{2(T-t(s))}}
\qquad s=s(t)=-\frac{1}{2}\log (T-t)
$$
and changing variables in formula~\eqref{f1}, we get
\begin{equation}\label{eq1000}
C\geq \int_{M} e^{-\frac{\vert
    y\vert^2}{2}}\,d\widetilde{\mu}_{-\frac{1}{2}\log T}\geq
\int_{-\frac{1}{2}\log T}^{+\infty}\int_{M} e^{-\frac{\vert y\vert^2}{2}}
\vert\widetilde{\HHH}+\langle
  y\,\vert\,\widetilde{\nu}\rangle\vert^2\,d\widetilde{\mu}_s(y)
\,ds\,.
\end{equation}

It follows that reasoning like in~\cite{huisk3} and~\cite{stone1}
(or~\cite{stone3}), if the singularity is of type I, the curvature of
the rescaled hypersurfaces $\widetilde{\varphi}_s:M\to\R^{n+1}$ is
uniformly bounded and any sequence converges (up to a subsequence) to
a limit embedded hypersurface $\widetilde{M}_\infty$ satisfying $\widetilde{\HHH}=-\langle
  x\,\vert\,\widetilde{\nu}\rangle$ which is the defining equation for a 
homothetic solution of MCF.\\
Moreover, By the estimates of Stone~\cite[Lemma~2.9]{stone1}, this limit
hypersurface satisfies
\begin{equation*}
\frac{1}{(2\pi)^{n/2}}\int_{\widetilde{M}_\infty} e^{-\frac{\vert
    y\vert^2}{2}}\,d{\mathcal{H}}^n(y)=\lim_{t\to
  T^-}\sigma(\varphi_t,T-t)=\Sigma>1\,.
\end{equation*}
Clearly, by this equation, this embedded limit hypersurface cannot be
empty. Moreover, it cannot be flat also, as it would be an hyperplane for the
origin of $\Ri$ (the only hyperplanes satisfying $\HHH=-\langle
  x\,\vert\,\nu\rangle$ must pass through the origin) as the above integral
  would be one.

\begin{prop} At a singular time $T$ of the MCF of an embedded compact
  hypersurface the quantity $\Sigma$ is larger than one.\\
  If the singularity of the flow is of type I, any sequence of rescaled
  hypersurfaces (with the maximal curvature) around the maximizer
  points for the Huisken's functional at times $t_i\nearrow T$
  converges, up to a subsequence, to a nonempty and nonflat, smooth embedded
  limit hypersurface, satisfying $\HHH=-\langle x\,\vert\,\nu\rangle$.
\end{prop}

Suppose now that we are dealing with the special case of an embedded
closed curve $\gamma_t$ evolving in the plane $\R^2$. Rescaling as
above, {\em without assuming anything} about the ``type'' of a 
singularity at some time $T$, 
we can extract a subsequence $\widetilde{\gamma}_{s_i}$ 
of rescaled curves such that: 
$$
\int_{\gamma_{s_i}} e^{-\frac{\vert y\vert^2}{2}}
\vert\widetilde{\HHH}+\langle
  y\,\vert\,\widetilde{\nu}\rangle\vert^2\,d\mathcal{H}^1(y)\to0
$$
with locally equibounded lengths. This implies that the curves
$\widetilde{\gamma}_{s_i}$ have also locally equibounded $L^2$ norms of
the curvature. Possibly passing to another subsequence (not relabeled)
we can assume that

\begin{itemize}
\item the curves $\widetilde{\gamma}_{s_i}$ converges in $C^1_{loc}$ 
to a limit curve $\widetilde{\gamma}_{\infty}$ 
with equibounded curvatures $\widetilde{k}$ in $L^2_{loc}$;

\item the curve $\widetilde{\gamma}_{\infty}$ satisfies
  $\widetilde{k}=-\langle x\,\vert\,\nu\rangle$ distributionally;

\item there holds $\frac{1}{(2\pi)^{n/2}}\int_{\widetilde{\gamma}_\infty} e^{-\frac{\vert
    y\vert^2}{2}}\,d{\mathcal{H}}^1(y)=\Sigma>1$;

\item finally, the curve $\widetilde{\gamma}_{\infty}$ is embedded,
  that is, without self--intersections, by the geometric argument of
  Huisken in~\cite{huisk2}.
\end{itemize}

By a bootstrap argument, using the condition $\widetilde{k}=-\langle
x\,\vert\,\nu\rangle$, it follows that $\widetilde{\gamma}_\infty$ 
is a smooth curve and since the only embedded curves in the plane
satisfying such condition are the lines through the origin and the unit
circle, $\widetilde{\gamma}_\infty$ has to be among them.\\
Then, the curve $\widetilde{\gamma}_\infty$ cannot be a line through
the origin, because for all of them the value of the integral $\frac{1}{(2\pi)^{n/2}}\int_{\widetilde{\gamma}_\infty} e^{-\frac{\vert
    y\vert^2}{2}}\,d{\mathcal{H}}^1(y)$ is one. Hence,
$\widetilde{\gamma}_\infty$ must be the unit circle.\\
This implies that at some time the curve $\gamma_t$ has become
$C^1$--close, hence a graph, over a round circle (in particular, it is
starshaped). It is then straightforward to see by means of maximum
principle that this last property is preserved during the evolution.\\
Then, by means of the interior estimates of Ecker and
Huisken~\cite{eckhui2}, one can find a close (in time) sequence of
rescaled curves converging in $C^2_{loc}$ to the unit
circle. Then, as this implies that at some time the curve has become
convex, the singularity can only be a type I {\em vanishing}
singularity. As a consequence, type II singularities for embedded
closed curves are not possible.

\begin{rem}
It would be very interesting if this argument can be extended in 
higher dimensions, that is, if rescaling the moving hypersurface
around the points maximizing the Huisken's functional could help to 
produce homothetic blowups also in the case of a Type II
singularity. Some results in this direction has been obtained 
by Ilmanen in~\cite{ilman3,ilman4}.
\end{rem}

\begin{ackn} Annibale Magni is partially supported by the ESF Programme ``Methods of
Integrable Systems, Geometry, Applied Mathematics'' (MISGAM) and Marie
Curie RTN ``European Network in Geometry, Mathematical Physics and
Applications'' (ENIGMA).
\end{ackn}

\bibliographystyle{amsplain}
\bibliography{monoton}

\end{document}